\documentclass[12pt,reqno,a4paper]{amsart}
\usepackage{amsfonts,hyperref}
\hypersetup{
    colorlinks=true,
    linkcolor=blue,
    citecolor=Green,
}
\usepackage[dvipsnames]{xcolor}
\usepackage{graphicx}
\usepackage[left=2.5cm,right=2.5cm,top=2.5cm,bottom=2.5cm,twoside]{geometry}

\newcommand{\veps}{\varepsilon}
\newcommand{\R}{\mathbb{R}}

\newcommand{\C}{\mathbb{C}}
\newcommand{\N}{\mathbb{N}}
\newcommand{\Z}{\mathbb{Z}}

\newcommand{\vp}{\varphi}

\newtheorem{theorem}{Theorem}[section]
\theoremstyle{plain}

\newtheorem{corollary}[theorem]{Corollary}

\newtheorem{exa}[theorem]{Example}
\newenvironment{example}{\begin{exa}\rm}{\end{exa}}

\newtheorem{lemma}[theorem]{Lemma}

{\theoremstyle{definition}
\newtheorem{remark}[theorem]{Remark}}

\numberwithin{equation}{section}

\newtheorem{thm}{Theorem}

\title[Non-linear differential equations
of Tumura-Clunie type]{On meromorphic solutions of non-linear differential equations
of Tumura-Clunie type}
\author[Heittokangas]{Janne Heittokangas}
\author[Latreuch]{Zinelaabidine Latreuch}
\author[Wang]{Jun Wang}
\author[Zemirni]{Mohamed Amine Zemirni}
\address[J.~Heittokangas]{University of eastern Finland, Department of Physics and Mathematics,
P.O.~Box 111, 80101 Joensuu, Finland}
\email{janne.heittokangas@uef.fi}
\address[Z.~Latreuch]{Department of Mathematics, Laboratory of Pure and Applied Mathematics, University of Mostaganem (UMAB)}
\email{z.latreuch@gmail.com}
\address[J.~Wang]{School of Mathematical Sciences, Fudan University, Shanghai, P. R. China}
\email{majwang@fudan.edu.cn }
\address[M.~A.~Zemirni]{University of eastern Finland, Department of Physics and Mathematics,
P.O.~Box 111, 80101 Joensuu, Finland}
\email{amine.zemirni@uef.fi}

\begin{document}

\begin{abstract}
Meromorphic solutions of non-linear differential equations of the form $f^n+P(z,f)=h$ 
are investigated, where $n\geq 2$ is an integer, $h$ is a meromorphic function, and  $P(z,f)$ is 
differential polynomial in $f$ and its derivatives with small functions as its coefficients.
In the existing literature this equation has been studied in the case when $h$ has the 
particular form $h(z)=p_1(z)e^{\alpha_1(z)}+p_2(z)e^{\alpha_2(z)}$,
where $p_1, p_2$ are small functions of $f$ and  $\alpha_1,\alpha_2$ are entire functions.
In such a case the order of $h$ is either a positive integer or equal to infinity.
In this article it is assumed that $h$ is a meromorphic solution of the linear differential equation
$h'' +r_1(z) h' +r_0(z) h=r_2(z)$ with rational coefficients $r_0,r_1,r_2$, and hence the order of
$h$ is a rational number. Recent results by Liao-Yang-Zhang (2013) and Liao (2015) follow
as special cases of the main results. 


\end{abstract}

\maketitle
\section{Introduction and main results}
\thispagestyle{empty}

\subsection{\sc Background}

Among the most famous complex non-linear differential equations are the ones introduced by 
Briot-Bouquet, Malmquist-Yosida, Painlev\'e, Riccati and Schwarz \cite{hille, L}. Differing
from linear differential equations with analytic coefficients, the meromorphic nature of solutions of non-linear equations is not given, let alone the analyticity of solutions. For example, it took over
a century to find a rigorous proof for the fact that all local solutions of the innocent looking
Painlev\'e's first equation $f''=z+6f^2$ can be analytically continued to single-valued
meromorphic solutions in the complex plane~\cite{HL}. Moreover, it is known that some Riccati
equations with meromorphic coefficients do not possess meromorphic solutions \cite[p.~172]{L}.
Along with the meromorphic nature of solutions, finding the growth and value distribution of
solutions have been studied for many equation types. At times a very general equation type
such as the Malmquist-Yosida equation $(y')^n=R(z,y)$ is shown to reduce to a lot more specific equations \cite[Ch.~10]{L}.  

Here we consider transcendental meromorphic solutions $f$ of non-linear differential equations
of Tumura-Clunie type
		\begin{equation}\label{e1}
		f^n + P(z,f)=h(z),
		\end{equation}
where $n\ge 2$, $h$ is a meromorphic function, and $P(z,f)$ is differential polynomial in
$f$ and its derivatives with small functions $a(z)$ as its coefficients. This means that $T(r,a)=S(r,f)$, where $S(r,f)$ denotes any quantity such that $S(r,f)=o(T(r,f))$ as $r\to\infty$ outside of a possible exceptional set of finite linear measure. 

The origin of Tumura-Clunie theory lies in the following theorem by Tumura \cite{Tumura}, published in 1937.
The proof was completed by Clunie \cite{Clunie} 25 years later.

\begin{thm}\label{Th.A}
	Let $f$ and $g$ be entire functions, and let
	$$
	h=a_nf^n +a_{n-1}f^{n-1}+\cdots+a_1f+a_0 \,\,\, (a_n \not\equiv 0)
	$$
where $a_j(z)$ $(j=0,1,\ldots,n)$ are small functions of $f$. If $h=be^g$, where $b(z)$ is a small
function of $f$, then
	$$
	h=a_n\left( f+\frac{a_{n-1}}{na_n}\right) ^n.
	$$
\end{thm}

As an extension to Theorem \ref{Th.A}, Hayman \cite[\S3.5]{H} proved the following result.

\begin{thm}\label{Th.B}
	Let $f$ be a non-constant meromorphic function in $\C$, and let
	$$
	h(z):=f^n+P(z,f),
	$$
	where $P(z,f)$ is a differential polynomial of degree at most $n-1$ in $f$ and its derivatives.
	If the coefficients of $P(z,f)$ are small functions of $f$, and if
	$$
	N(r,f)+N\left( r,\frac{1}{h}\right)  =S(r,f),
	$$
	then
	$
	h(z)=\left( f(z)+\frac{\alpha(z)}{n}\right) ^n,
	$
	 where $\alpha$ is a small function of $f$.
\end{thm}

Since \cite{H}, many theorems of Tumura-Clunie type have appeared in the literature,
see \cite{MS,R,Y,Yi} and the references therein. Typical applications of these theorems are
among differential polynomials and non-linear differential equations.

In all previous articles \cite{Li2,LYZ,Liao,LLW,YL} the equation \eqref{e1} has been studied
in the case when $h$ has the particular form
		\begin{equation}\label{fff}
		h(z)=p_1(z)e^{\alpha_1(z)}+p_2(z)e^{\alpha_2(z)},
		\end{equation}
where $p_1, p_2$ are small functions of $f$ and  $\alpha_1,\alpha_2$ are entire functions.
For example, Yang and Li \cite{YL} showed that the differential equation
	$$
	f^{3}+\frac{3}{4}f''=h(z),
	$$
where $h(z)=-\frac14\sin3z=-\frac{1}{8i}\left(e^{3iz}-e^{-3iz}\right)$,
has exactly three non-constant entire solutions: $f_1(z)=\sin z,\, f_2(z)=\frac{\sqrt{3}}{2}\cos z-\frac{1}{2} \sin z$ and $f_3(z)=-\frac{\sqrt{3}}{2}\cos z-\frac{1}{2} \sin z$.

In the case of \eqref{fff}, the order of $h$ is either an integer or equal to infinity.
In this article, we suppose that $h$ is a meromorphic solution of the equation
		\begin{equation}\label{e2}
		h'' +r_1(z) h' +r_0(z) h=r_2(z),
		\end{equation}
where $r_0,r_1$ and $r_2$ are rational functions. In this case, the order of $h$ can be a
non-integer, and $h$ can be a special function such as the Airy integral.
These cases have not been covered in the previous studies. 
An additional bonus of the use of \eqref{e2} is that it
allows us to shorten some of the calculations which appear in the proofs. 

The idea that a coefficient of a complex differential equation is a solution of another
differential equation is not new. For example, the solutions of
	$
	f''+Af'+Bf=0
	$
are studied in \cite{WLHQ, WW} under the assumption that $A$ solves the differential equation
$g''+Pg=0$, where $P$ is a polynomial.

The cases when
$h$ has $n$ exponential terms in \eqref{fff} or when \eqref{e2} is replaced with $n$th order
linear differential equation are still open for further research.

\subsection{\sc Notation}

The differential polynomial $P(z,f)$ in \eqref{e1} is a finite sum
	\[
		P(z,f)=\sum_{j=1}^{\ell}M_j[f],
		\]
where $M_j[f]=a_j f^{n_{0j}} (f')^{n_{1j}}\cdots(f^{(k)})^{n_{kj}}$ is a differential
monomial generated by $f$, the powers $n_{0j}, n_{1j},\dots, n_{kj}$ are non-negative integers,
and $a_j(z)$ is a small function
of $f$ (or, in particular, a rational function). The expressions
	$$
	\gamma_{M_{j}}=\sum\limits_{s=0}^{k}n_{sj}\quad\textnormal{and}\quad
	\Gamma_{M_{j}}=\sum\limits_{s=0}^{k}(s+1)n_{sj}
	$$
are called the degree of $M_j$ and the
weight of $M_j$, respectively. Then we say that $P(z,f)$ has the degree
$\gamma_P=\max\limits_{1\leq j\leq \ell}\gamma_{M_{j}}$ and the weight $\Gamma_P=\max\limits_{1\leq j\leq \ell}\Gamma_{M_{j}}$. The differential polynomial $P(z,f)$ is linear if $\gamma_P \le 1$, otherwise, it is non-linear.

Throughout the paper, we make use of the notation
	\begin{equation}\label{global}
	r_0(z)\sim C_{0}z^m\quad \textnormal{and}\quad r_1(z)\sim C_{1}z^l,\ \textnormal{as}\ z\to\infty,
	\end{equation}
where $r_0,r_1$ are the coefficients in \eqref{e2}, $C_0,C_1\in\C$, $C_0\neq 0$ and $m,l\in\Z$.

For $a\in\widehat{\C}$ and $k\in\N$, we denote the integrated  counting function of the $a$-points
of a meromorphic $f$ of multiplicity at most $k$ by $N_{k)}(r,a,f)$, each $a$-point being counted according to its multiplicity. Further, $N_{(k}(r,a,f)$ denotes the integrated counting function of the $a$-points of $f$
of multiplicity at least $k$, again each $a$-point being counted according to its multiplicity.
Then $N(r,a,f)=N_{k)}(r,a,f)+N_{(k+1}(r,a,f)$.
Finally, $\rho(f)$ denotes the order of $f$, while $\lambda(f)$ and $\lambda(1/f)$ denote
respectively the exponents of convergence of the zeros of $f$ and the poles of $f$.

\subsection{\sc Main results}

In this subsection we settle for stating our two main results regarding the meromorphic solutions $f$ of \eqref{e1}.
Examples of solutions of \eqref{e1} as well as corollaries and further discussions on our
main results can be found in the later sections. 

The first result shows that either $f$ has finitely many zeros and poles or else the number of zeros and poles of $f$ restrict the growth of $T(r,f)$ in a specific way. 

\begin{theorem}\label{theo1.0}
	Let $f$ and $h$ be  meromorphic solutions of \eqref{e1} and \eqref{e2}, respectively, and assume that
	$f$ is transcendental. Then one of the following holds:
	\begin{itemize}
	\item[(1)]  $\rho(h)$ is an integer, $f$ is of the form $f(z)=q(z)e^{\alpha (z)}$, where $q$ is a rational function, $\alpha$ is a non-constant polynomial, and
		\begin{equation}\label{characteristics}
		T(r,h)=nT(r,f)+S(r,f).
		\end{equation}
	Furthermore, if $r_1$ and $r_0$ are polynomials, then $q$ is a constant.
	
	\item[(2)] If $ n\ge j+1 $ and $ \gamma_P \le n-j $ for some integer $ j \ge 1 $, then
		$$
		T(r,f) \le \frac{2}{j} \overline{N} \left(r, \frac{1}{f} \right) + N(r,f) + \frac{2}{j} \overline{N} (r, f ) + S(r,f).
		$$
	\end{itemize}
\end{theorem}

Meromorphic solutions $f$ with few poles in the sense of $N(r,f)=S(r,f)$ are frequently studied
in the literature, not only for the equation \eqref{e1}. If $f$ is such a solution of \eqref{e1}
and if $j\geq 3$, then only the Case (1) in Theorem~\ref{theo1.0} can occur. The next result gives
a further account of this situation. In particular, $f$ may have a finite non-integer order $\rho(f)$,
in which case $\lambda(1/f)=\rho(f)$ holds.

\begin{theorem}\label{theo1.3}
Let $n\ge 3$, and let $f$ be a transcendental meromorphic solution of \eqref{e1}, where $h$ is a transcendental meromorphic solution of \eqref{e2}, $P(z,f)$ has rational coefficients, and $\gamma_P\le n-2$. If $N(r,f)=S(r,f)$, then $f$ has finitely many poles, $\rho(h)=\rho(f)$, and one of the following holds:
	\begin{itemize}
		\item[(1)] The conclusion of Theorem~\ref{theo1.0}(1) holds.
		
		\item[(2)] $T(r,f)=N_{1)}(r,1/f)+O(\log r)$, the function $f$ is of order $1+m/2$, and one of the following two situations for the parameters in \eqref{global} occur.
			\begin{enumerate}
			\item[(i)] We have $l\leq -1\leq m$. Moreover,
			$$
			\log M(r,f)=\frac{2\sqrt{\left| C_{0}\right| }}{n(m+2)}r^{1+m/2}(1+o(1)),\quad r\to\infty.
			$$
			\item[(ii)]  We have $m=2l\ge 0$ and
				$
				{C_0}=\frac{n(n-1)}{(2n-1)^2}C_1^2.
				$
			Moreover,
			$$
			\log M(r,f)= \frac{2\sqrt{\left|C_0\right|}}{\sqrt{n(n-1)}(m+2)}r^{1+{m}/{2}}(1+o(1)),\quad r\to\infty.
			$$
		\end{enumerate}
	\end{itemize}
\end{theorem}

\begin{remark}
(1) We will show in Section~\ref{corollaries-sec} that a consequence of Theorem~\ref{theo1.3} generalizes and improves
\cite[Theorem~1]{LYZ} and \cite[Theorem~1.7]{Liao}.

(2) From the proof of Theorem~\ref{theo1.3}, it follows in the sub-case (i) that $f$ satisfies a second-order differential equation
			\begin{equation*} 
			f'' +R(z) f'+S(z) f=0,
			\end{equation*}
where $R$ and $S$ are rational functions such that $|R(z)|\asymp |r_1(z)|$ and $|S(z)|\sim |r_0(z)|/ n^2 $, as $z\to \infty$.
In the sub-case (ii), $f$ satisfies a first order differential equation
	\[
			f'+S(z) f=Q,
			\]
where $S$ and $Q$ are non-zero rational functions and $|S(z)|\sim {|r_1(z)|}/(2n-1)$, as $z\to\infty$.
\end{remark}

\subsection{\sc Organization of the paper}
The rest of this paper is organized as follows. 
Corollaries of our main results can be found in Section~\ref{corollaries-sec}, where we find that  \cite[Theorem~1]{LYZ} and \cite[Theorem~1.7]{Liao} reduce to special cases of a corollary of Theorem~\ref{theo1.3}. Section~\ref{examples-sec} consists of examples of meromorphic solutions of \eqref{e1}, which illustrate many possibilities that can happen. In particular, solutions of finite non-integer order do occur. Section~\ref{lemmas-sec} consists of general auxiliary results from the existing literature, as well as on new lemmas on linear differential equations with rational coefficients. 
Finally, proofs for Theorems~\ref{theo1.0} and~\ref{theo1.3} can be found in Section~\ref{proofs-sec}, while proofs of their corollaries are in Section~\ref{proofs-cor-sec}.


\section{Corollaries}\label{corollaries-sec}


We proceed to express Theorem~\ref{theo1.0} in terms of the deficiencies
	\begin{equation*}
	\delta(a,f)=1-\limsup_{r\to\infty}\frac{N(r,a,f)}{T(r,f)}\quad\textnormal{and}\quad
	\Theta(a,f)=1-\limsup_{r\to\infty}\frac{\overline{N}(r,a,f)}{T(r,f)},
	\end{equation*}
familiar from the basic Nevanlinna theory. We have 	
	$
	0\leq\delta(a,f)\le \Theta(a,f)\le 1
	$
for any $a\in \widehat{\C}$. If the $a$-points of $f$ are simple (apart from
finitely many exceptions), then these two deficiencies are equal.

\begin{corollary}\label{theo1.1}
	Let $f$ and $h$ be  transcendental meromorphic solutions of \eqref{e1} and \eqref{e2},
	respectively. Suppose that one of the following holds:
	\begin{itemize}
	\item[(1)] $n\ge j+1$, $\gamma_P\le n-j$ and
			$
			\Theta(0,f)+\frac{j}{2} \delta(\infty,f)+\Theta(\infty,f)>2,
			$
	\item[(2)] $n\ge 6$ and $\Gamma_P\le n-5$.
	\end{itemize}
Then the conclusion of Theorem~\ref{theo1.0}(1) holds.
\end{corollary}

Case (1) in Corollary~\ref{theo1.1} can be simplified by replacing the $\Theta$'s with corresponding
$\delta$'s, but at the expense of weakening the result. Example~\ref{sharpness-of-cor} below
shows the sharpness in the sense that ''$>2$'' cannot be replaced with ''$=2$''.

Next we state a consequence of Theorem~\ref{theo1.3}.

\begin{corollary}\label{CC-corollary}
Let $ n\ge 3$, and let $f$ be a transcendental meromorphic solution of \eqref{e1}, where $P(z,f)$ has rational coefficients, $\gamma_P \le n-2$, and $h$ is of the form
	\begin{equation} \label{hee}
	h(z)=p_{1}(z)e^{\alpha_{1}(z)}+p_{2}(z)e^{\alpha_{2}(z)}=:h_1(z)+h_2(z).
	\end{equation}
Here we assume that $p_{1},p_{2}$ are rational functions such that 
$p_{1}p_{2}\not\equiv 0$, while $\alpha_{1},\alpha_{2}$ are non-constant polynomials normalized such that $\alpha_1(0)=0=\alpha_2(0)$.
Write
	$$
	\alpha_j(z)=a_jz^{s_j}+O\left(z^{s_j-1}\right), \quad j=1,2.
	$$
If $N(r,f)=S(r,f) $, then $f$ has finitely many poles, $\rho(f)= \rho(h)=s_1=s_2$, 
and $f$ takes one of the following forms:
	\begin{enumerate}
		\item 	
		$
		 f(z) = q(z)e^{\alpha(z)},
		$
		where $q$ is a non-zero rational function and $\alpha$ is a non-constant polynomial normalized with $\alpha(0)=0$. Moreover,
		the following conclusions hold. 
		\begin{enumerate}
		\item[(i)] If $a_1=a_2$, then $n\alpha=\alpha_1=\alpha_2$, $q^n=p_1+p_2$ and $P(z,f)\equiv 0$. In particular, if $p_1,p_2$ are polynomials, then $q$ is also a polynomial.
		\item[(ii)] The case when $|a_1|=|a_2|$ and $a_1\neq a_2$ is not possible. 
		\item[(iii)] If $|a_1|\neq |a_2|$, say $|a_1|>|a_2|$, then $n\alpha=\alpha_1$, $q^n=p_1$ and $P(z,f)\equiv h_2$.
		\end{enumerate}
		
		\item  
		$
		f(z) = q_1(z)e^{\beta(z)} +q_2(z) e^{-\beta(z)},
		$
		where $q_1,q_2$ are non-zero rational functions and $\beta$ is a non-constant polynomial such that $n\beta=\pm \alpha_1$ and $\alpha_1=-\alpha_2$.
				
		\item  
		$
		f(z)=q_1(z)e^{\frac{\alpha_1(z)+\alpha_2(z)}{2n-1}} +q_2(z),
		$
		 where $q_1,q_2$ are non-zero rational functions and $\max\{|a_1|,|a_2|\}/\min\{|a_1|,|a_2|\} = n /(n-1) $.
	\end{enumerate}
\end{corollary}

If $h_1,h_2$ are linearly independent, then Corollary~\ref{CC-corollary} shows that the assumption 
''$f$ has finitely many poles'' in \cite[Theorem~1]{LYZ} can be replaced with the much less
restrictive assumption ''$N(r,f)=S(r,f)$''.
The same observation applies to \cite[Theorem~1.7]{Liao} in the case when $h_1,h_2$ are linearly dependent.

As noted in \cite{LYZ}, a quick example of Case (3) is the solution $f(z)=e^z+z+1$ of the equation
	\begin{equation}\label{CC-example}
	f^3-2(z+1)^2f''-(z+1)^2f=e^{3z}+3(z+1)e^{2z}.
	\end{equation}
Some new examples of meromorphic solutions are given in Section~\ref{examples-sec}.
The proofs of Corollaries~\ref{theo1.1} and \ref{CC-corollary} are postponed to Section~\ref{proofs-cor-sec}.


\section{Examples of solutions}\label{examples-sec}


The first example proves the sharpness of Corollary~\ref{theo1.1} in the sense that the
assumption ''$>2$'' cannot be replaced with ''$=2$''.

\begin{example}\label{sharpness-of-cor} 
The meromorphic function $f(z)=e^{2z}/(e^z-1)$ has no zeros and it satisfies
		\begin{eqnarray*}
		T(r,f) &=& 2r/\pi + O(1),\\
		\overline{N}(r,f) &=& N(r,f) = r /\pi + O(1).
		\end{eqnarray*}
Thus $\Theta(\infty,f)=\delta(\infty,f)= 1/2$ and $\Theta(0,f)=1$. Moreover, $f$ solves the
equations
		\begin{eqnarray*}
		f^2 +f'-3f &=& e^{2z},\\
		f^3- \frac{1}{2} f'' + \frac{9}{2}f'-10 f &=& e^{3z}+3e^{2z}.
		\end{eqnarray*}
In the first case, we have $\Theta(0,f)+ \frac12 \delta(\infty,f)+\Theta(\infty,f)= 7/4<2$,
while in the second case $\Theta(0,f)+ \delta(\infty,f)+\Theta(\infty,f)= 2$. Here
$h_1(z)=e^{2z}$ and $h_2(z)=e^{3z}+3e^{2z}$ are solutions of the respective differential
equations $h''-4h=0$ and $h''-4h'+3h=0$.
\end{example}

The second example also illustrates the sharpness of Corollary~\ref{theo1.1}. Differing from
Example~\ref{sharpness-of-cor}, where the solution $f$ has infinitely many poles but no zeros, this time
$f$ has infinitely many poles and infinitely many zeros. 

\begin{example}\label{n=4-example}
The meromorphic function $f(z)= e^z +\frac{1}{e^z-1}$ satisfies
	$$
	T(r,f)=T\left( r, \frac{e^{2z}-e^z+1}{e^z-1}\right)=2 T(r,e^z)+O(1)= 2 r/\pi +O(1).
	$$
It is clear that all the poles of $f$ are simple, and hence $ \overline{N}(r,f)= N(r,f)=r/\pi +O(1)$. This gives us $\Theta(\infty,f)=\delta(\infty,f)=1/2$. The zeros of $f$ are precisely the zeros of the exponential sum
$g(z)=e^{2z}-e^z+1$ \cite{Stein}. If $z_0$ is a zero of $f$, then
$$
g'(z_0)=2e^{2z_0}-e^{z_0}=2\left(e^{2z_0}-e^{z_0}+1\right)+e^{z_0}-2=e^{z_0}-2\neq 0,
$$
which shows that the zeros of $f$ are all simple. Hence 
	$$
	\overline{N}(r,1/f)=N(r,1/f)=2r/\pi+O(1),
	$$ 
and then $\Theta (0,f)=0$. In addition, $f$ solves the equations
	\begin{eqnarray*}
	f^2 + f'-f-2&=&e^{2z},\\
	f^3 -\frac12 f''+\frac32 f'-4f-3&=&e^{3z},\\
	f^4+\frac{1}{6}f'''-f''+\frac{35}{6}f'-9f-10&=&e^{4z}+4e^{2z}.
	\end{eqnarray*}
In the first case $\Theta(0,f)+\frac12 \delta(\infty,f)+\Theta(\infty,f)=\frac34 \le 2$,
in the second case $\Theta(0,f)+ \delta(\infty,f)+\Theta(\infty,f)=1 \le 2$,
and in the third case $\Theta(0,f)+ \frac32\delta(\infty,f)+\Theta(\infty,f)=\frac54 \le 2$.
\end{example}

The third example shows that the condition $\gamma_P\le n-j$ in Corollary~\ref{theo1.1}(1) cannot  be weakened to $\gamma_P\le n-j+1$. The example also illustrates Corollary~\ref{CC-corollary}(2).

\begin{example}
Let $f(z)=e^{z/4}+e^{-z/4}$. It is clear that $\Theta(\infty,f)=\delta(\infty,f)=1$. We have
		$$
		T(r,f)= \frac{r}{2\pi}+O(1),\quad  N(r,1/f)=\frac{r}{2\pi}+O(1),
		$$
and all zeros of $ f $ are simple, hence $ \Theta(0,f)=\delta(0,f)=0 $. In addition, $f$ satisfies the non-linear differential equation
	$$
	f^{4}-64ff''+2=e^{z}+e^{-z},
	$$
and we have $\Theta(0,f)+ \delta(\infty,f)+\Theta(\infty,f)= 2$.
\end{example}

Next we show that $h$ can be rational in Theorem~\ref{theo1.0}.

\begin{example}
The meromorphic function $f(z)=\frac{1}{e^z-1}+z$ solves the equations
		\begin{eqnarray*}
		f^2+f'-(2z-1)f&=&-z^2+z-1,\\
		f^3-\frac{1}{2}f''+\left( 3z-\frac{3}{2}\right) f'-(3z^2-3z+1)f&=&-2z^3+3z^2+2z-\frac{3}{2},
		\end{eqnarray*}
where $h_1(z)=-z^2+z-1$ and $h_2(z)=-2z^3+3z^2+2z-\frac{3}{2}$ are rational solutions of the respective differential equations $h''-z/2h'+h=z/2-3$ and $h''-\frac{z}{3}h'+h=z^2-\frac{32}{3}z+\frac{9}{2}$.
\end{example}

Examples \ref{first}--\ref{last} below illustrate Theorem~\ref{theo1.3} in various ways.

\begin{example}\label{first}
Solutions of half-integer order exist, which shows that sub-case (i) may happen: The function $f(z)=\exp\left(\frac{z^{k/2}}{3}\right)+\exp\left(\frac{-z^{k/2}}{3}\right)$ is an entire solution of
	$$
	f^{3}+f''+\left( \frac{k}{2}-1\right) f'-\left( 3+\frac{k^{2}}{36}z^{k-2}\right) f=\exp\left({z^{\frac{k}{2}}}\right)+\exp\left({-z^{\frac{k}{2}}}\right),
	$$
	where $h=\exp\left({z^{k/2}}\right)+\exp\left({-z^{k/2}}\right)$ is an entire solution of
	$$
	h''-\frac{k-2}{2z}h'-\frac{k^{2}}{4}z^{k-2}h=0.
	$$
Moreover, $f$ satisfies
	$$
		f''-\frac{k-2}{2z}f'-\frac{k^{2}}{36}z^{k-2}f=0.
	$$
When $k$ is even, this also illustrates Case (2) in Corollary~\ref{CC-corollary}.
\end{example}

\begin{example}
Sub-case (ii) may happen: The function $f(z)=e^{z}+z+1$ is a solution of \eqref{CC-example}	and of $f'-f=-z$. We also see that $h(z)=e^{3z}+3(z+1)e^{2z}$ satisfies
	$$
	h''-\left( \frac{1}{z}+5\right) h'+3\left( \frac{1}{z}+2\right) h=0,
	$$
	where $\frac{C_0}{C_1^2}=\frac{6}{25}=\frac{3\cdot (3-1)}{(2\cdot 3-1)^2}$.
This also illustrates Case (3) in Corollary~\ref{CC-corollary}.	
\end{example}

\begin{example}
Solutions with finitely many poles exist: The function 
	$$
	\frac{z+1}{z-1} e^{z^2}
	$$
is a solution of 
	$$
	f^3 + \frac{z^2-5z+3}{2(4z^4-4z^3 -5z -1)} f'' - \frac{z^2-4z+4}{4z^4-4z^3 -5z -1} f' = \left(\frac{z+1}{z-1}\right)^3 e^{3z^2}+e^{z^2}. 
	$$
In particular, Case (3) in Corollary~\ref{CC-corollary} can happen.	
\end{example}

Sometimes Theorem~\ref{theo1.3} enables us also to prove the non-existence of meromorphic solutions
$f$ of \eqref{e1} satisfying $N(r,f)=S(r,f)$.

\begin{example}\label{last}
Suppose that the equation
	\begin{equation}
	f^{4}-\frac{1}{2}f^{\prime \prime }f-\frac{3}{4}f^{\prime \prime }+\frac{19}{%
	4}f^{\prime }-8f-9= \frac{7}{2}e^{2z}+e^{4z}  \label{q}
	\end{equation}
has a meromorphic solution $f$ such that $N(r,f)=S(r,f)$. Clearly $f$ must be transcendental. 
Moreover, we see that the 
right-hand side $h(z)=\frac{7}{2}e^{2z}+e^{4z}$ of \eqref{q} satisfies the equation
\[
h^{\prime \prime }-6h^{\prime }+8h=0,
\]
and the coefficients of this equation do not satisfy the conclusions in the
sub-cases (i) and (ii) of Theorem~\ref{theo1.3}(2). Thus $f$ must satisfy the
conclusion in Theorem~\ref{theo1.3}(1), that is, $f(z)=q(z)e^{az}$, where $q\not\equiv 0$ is rational function and $a\in\C\setminus\{0\}$. From \eqref{characteristics}, we have $|a|=1$. Substituting $f(z)=q(z)e^{az}$ into $\eqref{q}$, we find that the left-hand side of \eqref{q} is a polynomial in $e^{az}$ with rational coefficients, the dominant term being $q(z)^4e^{4az}$. Since $h(z)=e^{4z}(1+o(1))\to\infty$ exponentially in the right half-plane, we deduce that
$a=1$ and $q(z)\equiv 1$ must hold. Thus our substitution simplifies to   
	$$
	4e^{2z}+4e^z+9=0,\quad z\in\C,
	$$
which obviously cannot hold. Hence \eqref{q} possesses no meromorphic solutions satisfying $N(r,f)=S(r,f)$. By Theorem~\ref{theo1.3} all meromorphic solutions $f$ of \eqref{q} satisfy 
$\lambda \left( 1/f\right) =\rho \left( f\right)$. The function $f(z)=e^z + (e^z-1)^{-1}$ is one such a solution.
\end{example}

The case when $n=\gamma_P$ and $f$ has infinitely zeros or poles is not covered in our results.
The following example shows that this situation may occur. However, as of yet we don't have means
to expand the general theory to the case $n=\gamma_P$ since the standard tools such as Clunie's lemma 
would no longer be at our disposal. 

\begin{example}\label{Airy-ex}
It is well-known that the Airy integral $\operatorname{Ai}\,(z)$ is an entire solution of the equation $h''-zh=0$ \cite{olver}.
Thus it is easy to see that $f(z)=\operatorname{Ai}\,(z)$ solves the equation	
	$$
	f^n+z^{-1}f''-z^{-n}(f'')^n=\operatorname{Ai}\,(z)
	$$
with rational coefficients. It is well-known that 
$\rho(\operatorname{Ai})=\lambda(\operatorname{Ai})=3/2$. However, the frequency of zeros of 
$\operatorname{Ai}(z)$ is less than the usual amount in the sense of $\delta(0,\operatorname{Ai})=1/2$.
\end{example}


\section{Auxiliary results}\label{lemmas-sec}


\subsection{\sc General lemmas}

This sub-section contains general lemmas from the existing literature for proving our main results on non-linear differential equations.
\begin{lemma}[\cite{L}]\label{Mohon'ko lemma}
	Let $f$ be a meromorphic function. Then for all irreducible rational functions in $f$,
	$$
	R(z,f)=\left(\sum\limits_{i=0}^{p}a_{i}(z)f^{i}\right)\bigg/\left(\sum\limits_{j=0}^{q}b_{j}(z)f^{j}\right),
	$$
	with meromorphic coefficients $a_{i}(z), b_{j}(z)$, the characteristic function of $R(z,f)$ satisfies
	$$
	T(r,R)=dT(r,f)+O(\Psi(r))+S(r,f),
	$$
	where $d=\max\left\lbrace p,q\right\rbrace $ and $\Psi(r)=\max\left\lbrace T(r,a_{i}),T(r;b_{j})\right\rbrace $.
\end{lemma}

\begin{lemma}[\cite{L}]\label{Clunie}
	Let $f$ be a transcendental meromorphic solution of
	$$
	f^{n}Q^*(z,f)=Q(z,f),
	$$
	where $Q^*(z,f)$ and $Q(z,f)$ are polynomials in $f$ and its derivatives with meromorphic coefficients, say $\left\lbrace a_{\lambda}:\lambda \in I\right\rbrace $, such that $m(r,a_{\lambda})=S(r,f)$ for all $\lambda \in I$. If $\gamma_Q\leq n$, then
	$$
	m(r,Q^*(z,f))=S(r,f).
	$$
\end{lemma}

\begin{remark}
By using a similar method as in \cite[p.~40]{L}, we can prove that
	$$
	m(r,Q^*(z,f))=O(\log r),
	$$
when $Q^*(z,f)$ and $Q(z,f)$ are differential polynomials in $f$ with rational coefficients.
\end{remark}

\begin{lemma}[\cite{D}]\label{Clunie-Doeringer}
Let $f$ be a non-constant meromorphic function, and let $Q^*(z,f)$ and $Q(z,f)$ denote differential polynomials in $f$ with arbitrary meromorphic coefficients $q_{1}^{*}, q_{2}^{*},\ldots,q_{n}^{*}$ and $q_{1}, q_{2},\ldots,q_{l}$ respectively. Further, let $P$ be a non-constant polynomial of degree $p$
such that $\Gamma_Q \leq p$. If
	$$
	P(f)Q^{*}(z,f)=Q(z,f),
	$$
	then
	$$
	N(r,Q^{*}(z,f)) \leq \sum_{j=1}^{n}N(r,q_{j}^*) + \sum_{i=1}^{l}N(r,q_{j}) + O(1).
	$$
\end{lemma}

\begin{lemma}[\cite{Li2}]\label{liao}
	Suppose that $f(z)$ is a transcendental meromorphic function, and let $a(z),b(z),c(z),d(z)$
	be small functions of $f(z)$ such that $cd\not\equiv 0$. If
	\begin{equation*}
	af^{2}+bf^{\prime }f+c\left( f^{\prime }\right) ^{2}=d,
	\end{equation*}
	then
	\begin{equation*}
	c\left( b^{2}-4ac\right) \frac{d^{\prime }}{d}+b\left( b^{2}-4ac\right)
	-c\left( b^{2}-4ac\right) ^{\prime }+\left( b^{2}-4ac\right) c^{\prime
	}\equiv 0.
	\end{equation*}
\end{lemma}

Note that the condition $cd\not\equiv 0$ is written as $acd\not\equiv 0$ in \cite[Lemma~6, p.~313]{Li2}, but this is not necessary in the proof of Lemma~\ref{liao}. The assertion in Lemma~\ref{liao} still holds if  the functions $a(z),b(z),c(z),d(z)$ are small in the sense of $\max\left\lbrace \rho(a),\rho(b),\rho(c), \rho(d)\right\rbrace <\rho(f)$.

\subsection{\sc Lemmas on linear differential equations}

To prove our main results on non-linear differential equations, we need some new
updates and tweaks on known results about linear differential equations with rational coefficients. 
These results may also be of independent interest. Before proceeding, we remind the reader that, in the case of meromorphic coefficients, the solutions
may not always be meromorphic in $\C$. For example, the equation
$$
f''+2z^{-3}f'-z^{-4}f=0
$$
with rational coefficients possesses a non-meromorphic solution $f(z)=\exp\left(z^{-1}\right)$.
It may also happen that some solutions are meromorphic while others are not. For example,
the power series $f_1(z)=\cos\sqrt{z}$ and the function $f_2(z)=\sin \sqrt{z}$ are linearly
independent analytic solutions of
$$
f''+(2z)^{-1}f'+(4z)^{-1}f=0
$$
in the slit domain $\C\setminus\R_-$. In fact, $f_1$ is entire, while $f_2$ is non-meromorphic in $\C$.
Finally, all solutions can be entire even if the coefficients have poles, see Example~\ref{example}(b) below.

\begin{lemma}\label{11}
Let $f$ be a non-trivial meromorphic solution of
	\begin{equation*}
	f''+Af'+Bf=0,
	\end{equation*}
where $A$ and $B$ are meromorphic functions. Then
	\begin{align*}
	n_{(2}\left( r,\frac{1}{f}\right) &\le n_{1)}(r,A)+n_{(2)}(r,B),  \\
	n(r,f) & \le  n_{1)}(r,A)+n_{(2)}(r,B), 
	\end{align*}		
where $n_{(2)}(r,B)$ counts the number of double poles of $B$ two times. 
In particular, if $A$ and $B$ are entire functions, then $f$ is entire with only simple zeros.
\end{lemma}

\begin{proof}
We have
		$$
		1+ A \frac{f'}{f''}+ B \frac{f}{f''}=0.
		$$
If $z_0$ be a pole of any multiplicity or a multiple zero  of $f$,  then $f'/f''$ has a simple zero at $z_0$ and $f/f''$ has a double zero at $z_0$ in both cases. This means that $z_0$ must be either a simple pole of $A$  or a double pole of $B$ or both. Hence the conclusions hold.
\end{proof}
	
\begin{lemma}\label{l2}
	Let $R(z)$, $S(z)\not\equiv 0$ and $T(z)$ be rational functions such that $|R(z)|=O(|z|^{-1})$
	and $S(z)\sim C_Sz^m$ as $z\to\infty$, where $C_S \in \C$ and $m\in\Z$.
	Suppose that $f$ is a meromorphic solution of
	\begin{equation}\label{de-T}
	f''+R(z)f'+S(z)f=T(z)
	\end{equation}
	in $\C$. Then the following assertions hold.
	\begin{itemize}
		\item[\textnormal{(1)}] If $m\leq -2$, then $f$ is rational.
		\item[\textnormal{(2)}] If $m\geq -1$, then $f$ has at most finitely many poles and
		$$
		\log M(r,f)=\frac{2\sqrt{ |C_S| }}{m+2}r^{1+m/2}(1+o(1)),\quad r\to\infty.
		$$
	\end{itemize}
\end{lemma}

\begin{proof}
	(a) Suppose first that $f$ is transcendental entire. It is clear that
	\begin{equation}\label{logM}
	\lim_{r\to\infty}\frac{\log M(r,f)}{\log r}=\infty.
	\end{equation}
	Let $\nu(r)=\nu(r,f)$ and $\mu(r)=\mu(r,f)$ denote respectively the central index and the maximal
	term of $f$. Then the standard Wiman-Valiron theory gives us
	\begin{equation*} 
	f^{(k)}(z)=\left(\frac{\nu(r)}{z}\right)^k(1+o(1))f(z),\quad k\geq 0,
	\end{equation*}
	near the maximum modulus points of $f$, but outside of an exceptional set $E$ of finite logarithmic
	measure. Hence, dividing \eqref{de-T} by $f$, and then using  \eqref{logM}, we find that
	\begin{equation}\label{nu2}
	\nu(r)^2(1+o(1))+|C_S|r^{m+2}=o(1),
	\end{equation}
	as $r\to\infty$ with $r\not\in E$.
	If $m\leq -2$, we
	see from here that $\nu(r)$ is bounded, so that $f$ is a polynomial,
	which is a contradiction.
	Hence, from now on we assume that $m\geq -1$.
	It follows that the two leading terms $\nu(r)^2$ and $|C_S|r^{m+2}$
	in \eqref{nu2} must cancel out for $r$ large enough. Hence, if $\veps>0$, we may find a  constant $K=K(\veps)$ such that
	$$
	\nu(r)\leq (1+\veps)\sqrt{|C_S|}r^{1+m/2}\quad\textnormal{and}\quad \sqrt{|C_S|}r^{1+m/2}\leq (1+\veps)\nu(r)
	$$
	for all $r\geq K$ such that $r\not\in E$. Since $\nu(r)$ is non-decreasing and since $m\geq -1$,
	we may use a standard lemma from real analysis to avoid the exceptional set $E$. Now
	\begin{eqnarray*}
		\nu(r)&\leq& (1+\veps)^{2+m/2}\sqrt{|C_S|}r^{1+m/2}\leq (1+\veps)^k\sqrt{|C_S|}r^{1+m/2}\\
		\sqrt{|C_S|}r^{1+m/2}&\leq& (1+\veps)^{2+m/2}\nu(r)\leq (1+\veps)^k\nu(r)
	\end{eqnarray*}
	for all $r\geq K_1\geq K$, where $k$ is the smallest integer $\geq 2+m/2\geq 3/2$.
	Since
	$$
	(1+\veps)^k=1+\sum_{j=1}^k{k\choose j}\veps^j=1+\veps',
	$$
	we may write
	$$
	(1-\veps')\sqrt{|C_S|}r^{1+m/2}\leq \frac{\sqrt{|C_S|}r^{1+m/2}}{1+\veps'}\leq
	\nu(r)\leq (1+\veps')\sqrt{|C_S|}r^{1+m/2},\quad r\geq K_1.
	$$
	By using the well-known identity
	$$
	\log \mu(r)-\log \mu(K_1)=\int_{K_1}^r\frac{\nu(t)}{t}\, dt,
	$$
	we deduce that
	\begin{equation}\label{mu}
	(1-2\veps')\frac{2\sqrt{|C_S|}}{m+2}r^{1+m/2}\leq\log\mu(r)\leq (1+2\veps')\frac{2\sqrt{|C_S|}}{m+2}r^{1+m/2},\quad r\geq K_2\geq K_1.
	\end{equation}
	By using another well-known identity (\cite[Satz 4.4]{jank-volkmann})
	$$
	\mu(r)\leq M(r,f)\leq\mu(r)\left(\nu(\rho)+\frac{\rho}{\rho-r}\right),\quad 0<r<\rho,
	$$
	for $\rho=2r$, we obtain from \eqref{mu} that
	$$
	(1-3\veps')\frac{2\sqrt{|C_S|}}{m+2}r^{1+m/2}\leq \log M(r,f)\leq (1+3\veps')\frac{2\sqrt{|C_S|}}{m+2}r^{1+m/2},\quad r\geq K_3\geq K_2.
	$$
	Since the above reasoning works for any $\veps>0$, we deduce that
	$$
	\log M(r,f)=\frac{2\sqrt{|C_S|}}{m+2}r^{1+m/2}(1+o(1)),\quad r\to\infty.
	$$
	This completes the proof in the case when $f$ is entire.
	
	\medskip
	(b) Suppose then that $f$ is meromorphic. By a simple inspection, all poles of $f$ must be among
	the poles of $R(z)$, $S(z)$ and $T(z)$,
	thus $f$ has at most finitely many poles. Hence we may write $f=g/\Pi$, where $g$ is
	an entire function and $\Pi$ is a non-constant polynomial
	having zeros precisely the poles of $f$. By substituting
	$f$ into \eqref{de-T}, we find that $g$ solves a differential equation
	\begin{equation}\label{DE-g}
	g''+R^*(z)g'+S^*(z)g=\Pi(z)T(z),
	\end{equation}
	where
	\begin{eqnarray*}
		R^*(z) &=& R(z)-2\frac{\Pi'(z)}{\Pi(z)}\\
		S^*(z) &=& 2\left(\frac{\Pi'(z)}{\Pi(z)}\right)^2
		-\frac{\Pi''(z)}{\Pi(z)}-R(z)\frac{\Pi'(z)}{\Pi(z)}+S(z).
	\end{eqnarray*}
	It is easy to see that $|R^*(z)|=O\left(|z|^{-1}\right)$. If $m\leq -2$, then clearly
	$|S^*(z)|=O\left(|z|^{-2}\right)$, while if $m\geq -1$, then $S^*(z)\sim C_Sz^m$ as $z\to\infty$.
	Next we follow the reasoning in Part~(a) to get the assertions for $g$, and then for $f$.
	This completes the proof in the case when $f$ is meromorphic.
\end{proof}

\begin{lemma}\label{l3}
	Let $S(z)\not\equiv 0$ and $T(z)$ be rational functions such that $S(z)\sim C_{S} z^{m}$ as $z \rightarrow \infty$, where $C_{S} \in \C$ and
	$m\in \mathbb{Z}$. Suppose that $f$ is a meromorphic solution of
	\begin{equation*} 
	f^{\prime }+S(z)f=T(z).
	\end{equation*}%
	Then the following holds.
	\begin{itemize}
		\item[\textnormal{(a)}] If $m\leq -1$, then $f$ is rational.
		\item[\textnormal{(b)}] If $m\geq 0$, then $f$ has at most finitely many poles and
		$$
		\log M(r,f)=\frac{|C_S|}{m+1}r^{1+m}(1+o(1)),\quad r\to\infty.
		$$
	\end{itemize}
\end{lemma}

The proof of Lemma~\ref{l3} has very similar components as that of Lemma~\ref{l2}, and hence we omit the details. The next result is a second order case of a higher order result in Jank-Volkmann's book \cite[Satz 22.1, p. 208]{jank-volkmann}. Differing from \cite {jank-volkmann}, this result gives a concrete list of possible orders as well as possible maximum modulus types for $f$.

\begin{lemma}\label{possible-orders}
	Let $f$ be a transcendental meromorphic solution of
	\begin{equation}\label{RST2}
	f''+R(z)f'+S(z)f=T(z),
	\end{equation}
	where the coefficients $R(z)\not\equiv 0,S(z)\not\equiv 0,T(z)$ are rational functions such that $R(z)\sim C_Rz^n$
	and $S(z)\sim C_Sz^m$ as $z\to\infty$, where $C_R,C_S\in\C$
	and $n,m\in\Z$. Then $f$ has at most finitely many poles and
	$$
	\log M(r,f)=Cr^{\rho}(1+o(1)),\quad r\to\infty,
	$$
	with the following possibilities for $C$ and $\rho$:
	\begin{itemize}
		\item[\textnormal{(1)}] If $m>2n$, then $\rho=1+\frac{m}{2}\geq 1/2$ and $C=\frac{2\sqrt{|C_S|}}{m+2}$.
		\item[\textnormal{(2)}] If $n\leq m<2n$, then we have two possibilities:
		\begin{itemize}
			\item[\textnormal{(i)}] $\rho=n+1\geq 1$ and $C=\frac{|C_R|}{1+n}$,
			\item[\textnormal{(ii)}] $\rho=1+m-n\geq 1$ and $C=\frac{|C_S|}{(1+m-n)|C_R|}$.
		\end{itemize}
		\item[\textnormal{(3)}] If $m<n$, then $\rho=1+n\geq 1$ and
		$C=\frac{|C_R|}{1+n}$.
		\item[\textnormal{(4)}] If $m=2n$, then $\rho=1+n\geq 1$ and $C=\frac{\left| X\right| }{1+n}$, where $X$ is a complex solution of the quadratic equation
		$
		X^{2}+C_{R}X+C_{S}=0.
		$
	\end{itemize}
If $R(z)\equiv0$, then only Case \textnormal{(1)} is possible. In all cases, $\rho\geq 1/2$.
\end{lemma}

\begin{proof}
	The proof is very similar to that of Lemma~\ref{l2}, and hence we confine ourselves
	to sketch the key steps of the proof only.
	
	\medskip
	(a) Suppose first that $f$ is transcendental entire. Then we may use \eqref{logM}.
	Dividing \eqref{RST2} by $f$ and using Wiman-Valiron theory near the maximum modulus points of $f$, we obtain
	\begin{equation}\label{WV2}
	\nu(r)^2(1+o(1))+C_Rz^{1+n}\nu(r)(1+o(1))+C_Sz^{2+m}=o(1),\quad r\not\in E,
	\end{equation}
	where $r=|z|$ and the set $E\subset [1,\infty)$ has finite logarithmic measure.
	At least two of these three terms must be maximal and they must cancel out for $r$ large enough.
	We take for granted that $\nu(r)$ is asymptotically comparable to $r^\alpha$ for some $\alpha>0$,
	see \cite[p. 108]{valiron} or \cite[pp. 204-208]{jank-volkmann}. There are four possible cases as follows.
	\begin{itemize}
		\item[(i)] Suppose that the first and the second terms in \eqref{WV2} are maximal. Plugging
		in $r^\alpha$ in place of $\nu(r)$ and comparing the exponents, we find that
		$$
		2\alpha=\alpha+n+1,\quad 2\alpha>m+2,\quad \alpha+n+1>m+2.
		$$
		These simplify to $\alpha=n+1$ and $m<2n$. Here $n\geq 0$ must hold, for otherwise
		$\nu(r)$ is bounded, and hence $f$ is a polynomial. A more careful analysis of \eqref{WV2}
		reveals that $\nu(r)=|C_R|r^{1+n}(1+o(1))$, which gives us
		$$
		\log M(r,f)=\frac{|C_R|}{1+n}r^{1+n}(1+o(1)).
		$$
		\item[(ii)] Suppose that the first and the third terms in \eqref{WV2} are maximal. Similarly as in
		Case (i), we conclude that $\alpha=1+m/2$ and $m>2n$, where $m\geq -1$ must hold.
		In addition, $\nu(r)^2=|C_S|r^{2+m}(1+o(1))$, which gives us
		$$
		\log M(r,f)=\frac{2\sqrt{|C_S|}}{m+2}r^{1+m/2}(1+o(1)).
		$$
		\item[(iii)] Suppose that the second and the third terms in \eqref{WV2} are maximal. In this
		case $\alpha=1+m-n$, where $n\leq m<2n$. In addition,
		$|C_R|r^{1+n}\nu(r)=|C_S|r^{2+m}(1+o(1))$, which gives us
		$$
		\log M(r,f)=\frac{|C_S|}{(1+m-n)|C_R|}r^{1+m-n}(1+o(1)).
		$$
		\item[(iv)] Finally suppose that all three terms in \eqref{WV2} are maximal. Now
		$\alpha=n+1=1+m/2=1+m-n$, which is possible only if $m=2n$. Thus $\nu(r)=Cr^{1+n}(1+o(1))$, as $z \to \infty$. Substituting this into \eqref{WV2}, we obtain
		$$
		C^{2}r^{2(1+n)}(1+o(1))+C_RCr^{1+n}z^{1+n}(1+o(1))+C_Sz^{2(1+n)}=o(1).
		$$
		Taking $z=re^{it}$ where $t$ is chosen so that $z$ is near to the maximum modulus point of $f$, this yields
		
		\begin{equation}\label{WV3}
    	C^{2}(1+o(1))+C_RCe^{i(n+1)t}(1+o(1))+C_Se^{2i(1+n)t}=o(1).
		\end{equation}
	Since $C>0$, we can put $X=Ce^{-i(n+1)t}$, and by substituting this into \eqref{WV3}, we get
	$$
	X^{2}(1+o(1))+C_RX(1+o(1))+C_S=o(1),
	$$
	which is asymptotically equivalent to
	$
		X^{2}+C_RX+C_S=0.
	$
	\end{itemize}

	If $R(z)\equiv 0$, then the middle term in \eqref{WV2} disappears, and only Case (ii)
	can happen. This completes the proof in the case when $f$ is entire.
	
	\medskip
	(b) Suppose then than $f$ is meromorphic. Clearly $f$ has at most finitely many poles, so we
	may write $f=g/\Pi$, where $g$ is entire and $\Pi$ is a polynomial having zeros precisely
	at the poles of $f$. Substituting $f$ into \eqref{RST2}, we see that $g$ satisfies \eqref{DE-g}
	to which the reasoning in Part~(a) applies. This completes the proof in the case when $f$ is
	meromorphic.
\end{proof}

In the case $n\leq m<2n$ there are two possible orders for the solutions of \eqref{RST2}.
The following example shows that sometimes only one of the two possible orders actually
occur, and sometimes both occur.

\begin{example}\label{example}
	(a) In the equation
	$$
	f''+z^3f'-(z^4+z^2)f=0,
	$$
	we have $n=3$ and $m=4$, and hence $n\leq m<2n$. The possible orders are $n+1=4$ and $1+m-n=2$.
	However, all solutions $f\not\equiv 0$ of this equation have order 4, see \cite[Example~8]{GSW}.
	
	(b) In the equation
	$$
	f''-\frac{9z^4-4z^2+6z-2}{3z^2-2z}f'+\frac{18z^4-12z^3+6z}{3z-2}f=0,
	$$
	we have $n=2$ and $m=3$, and hence $n\leq m<2n$. The possible orders are $n+1=3$ and $1+m-n=2$.
	Since $f_1(z)=\exp\left(z^2\right)$ and $f_2(z)=\exp\left(z^3\right)$ are linearly independent
	solutions, both possible orders occur.
\end{example}

\begin{example}
	The functions $f_1(z)=e^{\sqrt{3}z^{2}}$ and $f_2(z)=e^{z^{2}}$ are linearly independent solutions of
	$$
	f''-\left(2\left(\sqrt{3}+1\right)z+z^{-1}\right)f'+4\sqrt{3}z^{2}f=0.
	$$
We have $n=1$ and $m=2$, and the only possible order is $n+1=2$. We remark also that
	$$
	C_R=-2\left(\sqrt{3}+1\right)\quad\textnormal{and}\quad
	C_S=4\sqrt{3}.
	$$
	On the other hand, $2$ and $2\sqrt{3}$ are solutions of the quadratic equation $X^{2}+C_RX+C_S=0$. This shows us that both of the possible types $1$ and $\sqrt{3}$ occur. 
\end{example}


\section{Proofs of theorems}\label{proofs-sec}


Throughout this section, $P'(z,f)$ and $P''(z,f)$ mean $\partial P / \partial z$ and $\partial^2 P/ \partial z^2$, respectively.

\subsection{\sc Proof of Theorem \ref{theo1.0}}
	
By differentiating \eqref{e1} twice, we get
	\begin{eqnarray}
	h^{\prime }&=& nf^{\prime }f^{n-1}+P^{\prime }\left( z,f\right), \label{e3} \\
	h^{\prime \prime }&=& nf^{\prime \prime }f^{n-1}+n\left( n-1\right) \left( f^{\prime }\right)
	^{2}f^{n-2}+P^{\prime \prime }\left( z,f\right). \label{e4}
	\end{eqnarray}
Substituting \eqref{e1}, \eqref{e3} and \eqref{e4} into \eqref{e2}, we obtain
	\begin{eqnarray}
	f^{n-2}\varphi &=& Q\left( z,f\right),\label{e5}
	\end{eqnarray}
where
	\begin{eqnarray}
	\varphi &=& r_{0}f^{2}+nr_{1}f'f+nf''f+n(n-1)(f')^2,\label{e7}\\
	Q(z,f)&=& -(P''(z,f)+r_{1}P'(z,f)+r_{0}P(z,f)-r_{2}).\nonumber 
	\end{eqnarray}
	
(1) If $\varphi\equiv 0$ then $f^n$ satisfies
	\begin{equation}
	(f^n)''+ r_1(z) (f^n)' +r_0 (z)f^n=0. \label{e13}
	\end{equation}
Clearly,  $f$ has a finite order. Since $f^n$ does not have simple zeros, and since $r_1$, $r_2$ are rational functions, the function $f^n$ has finitely many zeros and finitely many poles by Lemma~\ref{11}. The same is true for $f$. Thus  $f(z)=q(z)e^{\alpha(z)}$, where $q$ is a rational function and $\alpha$ is non-constant polynomial.  
By substituting this into the equation \eqref{e1}, we get
	\begin{equation}\label{reduced-eqn}
	q^n e^{n\alpha}+\sum_{j=0}^{\gamma_P}\beta_j e^{j\alpha}=h,
	\end{equation}
where $\beta_j$ are small functions of $f$ depending on the coefficients of $P(z,f)$, $ \alpha $ and $ q $. Then by Lemma~\ref{Mohon'ko lemma}, we have
	$$
	T(r,h)=nT(r,e^\alpha)+S(r,f)=nT(r,f)+S(r,f).
	$$
From this it follows that the order of $h$ is a positive integer.
Furthermore, if $ r_0 $ and $ r_1 $ are entire, then Lemma~\ref{11} and \eqref{e13} show that $ f^n $ has only simple zeros. Since $ n\ge 2 $, it follows that $ f $ does not have zeros, and hence $ q $ must be a constant.

(2) Now, assume that $\varphi\not\equiv 0$. Set, for $ j=1,2,3, \dots $,
		\begin{equation*}
		\vp_j = f^{j-2} \vp = r_{0}f^{j}+nr_{1}f'f^{j-1}+nf''f^{j-1}+n(n-1) \left(\frac{f'}{f} \right)^2 f^j.
		\end{equation*}
and then \eqref{e5} can be re-written as 	
	\begin{equation}\label{phij}
	f^{n-j} \varphi_j = Q(z,f), \quad j=1,2,3, \dots 
	\end{equation}
Since $ \gamma_Q =\gamma_P \le n-j $, and since all the coefficients of $ \vp_j  $ are of small proximity function of $f$, it follows from Lemma~\ref{Clunie} that 
	\begin{equation*} 
	m(r,\varphi_j)=S(r,f),\quad j=1,2,3,\dots.
	\end{equation*}
It follows from this, using the lemma on the logarithmic derivative and the first main theorem, that 
	\begin{equation} \label{mphij}
	j m(r,1/f) \le m(r,1/\varphi_j)+S(r,f) =N(r,\varphi_j)- N(r, 1/\varphi_j) +S(r,f).
	\end{equation}
Suppose that $f$ has a pole at $z_{0}$ of multiplicity $k$, which is not a zero or a pole of $r_{i}(z)$ $(i=0,1)$. Then the Laurent expansion of $f$ in the neighbourhood of $z_0$ is
	\begin{equation}\label{Laurent}
	f(z)=\frac{a_{k}}{(z-z_{0})^{k}}+\frac{a_{k-1}}{(z-z_{0})^{k-1}}+\cdots,\quad  a_{k}\neq 0.
	\end{equation}
Then, for all $z$ near $z_0$, we have
	$$
	nf''(z)f(z)^{j-1}+n(n-1)\left(\frac{f'(z)}{f(z)}\right)^2 f(z)^j =\frac{(nk+1)nka_k^j}{(z-z_0)^{jk+2}}+O\left(\left(\frac{1}{z-z_0}\right)^{jk+1}\right),
	$$
where $(nk+1)nka_k^j\neq 0$. Since the terms $r_0f^j$ and $nr_1f'f^{j-1}$ have a pole of multiplicity $\leq jk+1$
at $z_0$, it follows that $z_0$ is a pole of multiplicity $jk+2$ of $\varphi_j$. In other words,
any pole of multiplicity $k$ of $f$ is a pole of multiplicity $jk+2$ of $\varphi_j$, with at most
finitely many exceptions. Also, when $ j=1 $, we notice that any simple zero of $f$ is a pole of $\vp_1$. The remaining possible poles of $\varphi_j$ come from the poles of $r_0$ and $r_1$. Thus
		\begin{equation}	\label{pphij}	
		N(r,\varphi_j)=
		\left\{
		\begin{array}{lll}
		N_{1)} (r,1/f)+N(r,f)+ 2\overline{N}(r,f)+O(\log r), & & j=1,\\
		j N(r,f)+ 2\overline{N}(r,f)+O(\log r), & & j\ge 2.
		\end{array}
		\right.
		\end{equation}

The proof is now concluded in three cases: $j=1$, $j=2$ and $j\geq 3$.

(i) Suppose that $j=1$. If $f$ has a zero of multiplicity $ k\ge 3$ at $z_0$, which is not a pole or a zero of $r_i(z)$ ($ i=0,1 $), then the Taylor expansion of $ f $ in the neighborhood of $ z_0 $ is
		\begin{equation*}\label{taylor}	
		f(z)=a_k (z-z_0)^k + \cdots,\quad a_k\neq 0.
		\end{equation*}
Then, for all $ z $ near $z_0$, we have
		$$
		nf''(z)+n(n-1)\frac{f'(z)^2}{f(z)}= (nk-1)nk a_k (z-z_0)^{k-2} +\cdots ,
		$$
where $(nk-1)nk a_k\neq 0$. Since the terms $r_0f$ and $nr_1f'$ have a zero of multiplicity $\ge k-1$ at $z_0$, it follows that $z_0$ is a zero of $\vp_1$ of multiplicity $k-2$. In other words, any zero of $f$ of multiplicity $k\ge 3$ is a zero of $\varphi_1$ of multiplicity $k-2$, with at most finitely many exceptions. Thus,
	\begin{equation}\label{zphi1}
	N_{(3}(r,1/f)-2\overline{N}_{(3}(r,1/f) \le  N(r,1/\varphi_1)+O(\log r).
	\end{equation}

Using \eqref{mphij}, \eqref{pphij}, \eqref{zphi1} and taking into the account that $N(r,1/f)= N_{2)}(r,1/f) + N_{(3} (r,1/f)$, we get
		\begin{eqnarray*}
		T(r,f) &=& N_{2)}(r,1/f) + N_{(3} (r,1/f) +m(r,1/f) + O(1)\\
		&\leq& N_{2)}(r,1/f) + 2\overline{N}_{(3} (r,1/f)+N(r,1/\varphi_1) +m(r,1/f) + O(\log r)\\
				&\le&  N_{1)} (r,1/f)+N_{2)}(r,1/f) +2\overline{N}_{(3}(r,1/f)+ N(r,f)+ 2\overline{N}(r,f)+S(r,f)\\
				&=& 2\overline{N}(r,1/f)+ N(r,f)+ 2\overline{N}(r,f)+S(r,f).
		\end{eqnarray*}
This proves Theorem~\ref{theo1.0}(2) when $ j=1 $.

(ii) Suppose that $j=2$. Analogously as in (i), we see that any zero of $f$ of multiplicity $k\ge 2$ is a zero of $\varphi_2$ of multiplicity $2k-2$, with at most finitely many exceptions. Hence
		\begin{equation}\label{zd}
		2N_{(2}(r,1/f)-2\overline{N}_{(2}(r,1/f) \le N(r,1/\varphi_2)+O(\log r).
		\end{equation}

Using \eqref{mphij}, \eqref{pphij}, \eqref{zd} and since $N(r,1/f)= N_{1)}(r,1/f) + N_{(2} (r,1/f)$, it follows  that	
		\begin{equation}\label{7 0}
		\begin{split}
		2T(r,f)	&= 2N(r,1/f)+2m(r,1/f)+O(1)\\
		&\le 2{N}_{1)}(r,1/f)+2\overline{N}_{(2}(r,1/f)+N(r,\varphi_2)+S(r,f) \\
		&\le 2\overline{N}(r,1/f)+2N(r,f)+2\overline{N}(r,f) + S(r,f).
		\end{split}
		\end{equation}
This proves Theorem~\ref{theo1.0}(2) when $ j=2$.
		
(iii) Finally suppose that $j\ge 3$. Using the Taylor series argument, we notice that any zero of $f$ of multiplicity $k$ is a zero of $\varphi_j$ of multiplicity $jk-2$, with at most finitely many exceptions. Thus
		\begin{equation*}
		jN(r,1/f) -2 \overline{N}(r,1/f)\le N(r,1/\varphi_j)+O(\log r).
		\end{equation*}
Using this together with \eqref{mphij} and \eqref{pphij}, we get 
		\begin{eqnarray*}
		jT(r,f) &=& jN(r,1/f)+jm(r,1/f)+O(1)\\
		&\leq & 2\overline{N}(r,1/f)+N(r,\varphi_j)+S(r,f)\\
		& \le & 2\overline{N}(r,1/f)+jN(r,f)+2\overline{N}(r,f) + S(r,f),
		\end{eqnarray*}		
which proves Theorem~\ref{theo1.0}(2) when $ j\ge 3$.

\subsection{\sc Proof of Theorem \ref{theo1.3}: The case of finitely many poles}\label{Case1}

We prove the assertions under the assumption that $f$ has finitely many poles.
The case of infinitely many poles will be discussed in Section~\ref{Case2} below.

First, we prove that $\rho(f)=\rho(h)$.
On one hand, it is clear from \eqref{e1} that $\rho(h)\le \rho(f)$. On the other hand,
	\begin{equation}\label{growth}
	\begin{split}
	nT(r,f)&=T(r,f^n) \le T(r,h)+T(r,P(z,f))+ O(\log r)\\
	& \le T(r,h)+\gamma_P m(r,f)+\Gamma_P N(r,f)+O(\log r).
	\end{split}
	\end{equation}
Since $f$ has finitely many poles and $\gamma_P\le n-2$, it follows that $ T(r,f)\le O(T(r,h)+\log r ) $, as $r\to\infty$, and so $\rho(f)\le \rho(h)$.

Let $\varphi$ be the function defined by \eqref{e7}.
Since $\gamma_Q \le n-2$ and since $f$ has finitely many poles, an application of Lemma~\ref{Clunie} on the equation \eqref{e5} results in $T(r,\varphi)=O(\log r)$. Next, we discuss two cases.

\medskip
(1) If $\varphi\equiv 0$, then from the proof of Theorem~\ref{theo1.0}, we 
obtain the conclusion in Theorem~\ref{theo1.3}(1).

\medskip	
(2) Suppose that $\varphi\not\equiv 0$. Using the facts that $f$ is of finite order and has finitely many poles, it follows from \eqref{7 0} that
	\begin{equation*}
	T\left( r,f\right) =\overline{N}(r,1/f) +O(\log r).
	\end{equation*}
Also, from \eqref{zd} with $\varphi_2=\varphi$,  we obtain
		\begin{eqnarray*}
		\overline{N}_{(2}(r,1/f) &\le& 2N_{(2}(r,1/f)-2\overline{N}_{(2}(r,1/f) \le N(r,1/\varphi)+O(\log r)\\
		&\le& T(r,\varphi)+O(\log r) =O (\log r).
		\end{eqnarray*}
Thus,
		$$
		T(r,f)=N_{1)}(r,1/f)+O(\log r).
		$$
This proves one aspect of Theorem~\ref{theo1.3}(2).
	
In order to prove the two additional possibilities in Theorem~\ref{theo1.3}(2), we will make
some preparations. Differentiating \eqref{e7}, we get
	\begin{align}
	\varphi ^{\prime }= &r_{0}^{\prime }f^{2}+\left( 2r_{0}+nr_{1}^{\prime
	}\right) f^{\prime }f+nr_{1}\left( f^{\prime }\right) ^{2}+nr_{1}f^{\prime
		\prime }f+nf^{\prime \prime \prime }f+n\left( 2n-1\right) f^{\prime \prime
	}f^{\prime }. \label{e19}
	\end{align}
	If $z_{0}$ is a simple zero of $f$ that is not a zero or a pole of $r_{0}, r_{1}$ and $\varphi $, then it follows from \eqref{e7} and \eqref{e19} that
	$$
	\varphi(z_0)=n(n-1)f'(z_0)^2\quad\textnormal{and}\quad \varphi'(z_0)= nr_1(z_0) f'(z_0)^2+n(2n-1)f''(z_0)f'(z_0).
	$$
Thus $z_{0}$ is a zero of $\left( r_{1}\varphi -\left(
	n-1\right) \varphi ^{\prime }\right) f^{\prime }+\left( 2n-1\right) \varphi
	f^{\prime \prime }.$
	Define the meromorphic function
	\begin{equation}\label{e20}
	\psi =\frac{\left( r_{1}\varphi -\left( n-1\right) \varphi ^{\prime }\right)
		f^{\prime }+\left( 2n-1\right) \varphi f^{\prime \prime }}{f}.
	\end{equation}
Then $T\left( r,\psi \right) =O\left( \log r\right) $ and
	\begin{equation}
	f^{\prime \prime }=\left( \frac{n-1}{2n-1}\frac{\varphi ^{\prime }}{\varphi }%
	-\frac{1}{2n-1}r_{1}\right) f^{\prime }+\frac{\psi }{\left( 2n-1\right)
		\varphi }f.  \label{e21}
	\end{equation}
By substituting this into \eqref{e7}, we get
	\begin{equation}\label{e22}
	\varphi(z) =a(z) f^{2}+b(z) f^{\prime}f+c\left( f^{\prime }\right) ^{2},  
	\end{equation}
where
	\begin{equation} \label{abc}
	a\left( z\right) =r_{0}+\frac{n}{2n-1}\frac{\psi }{\varphi },\text{ }b\left(
	z\right) =\frac{n\left( n-1\right) }{2n-1}\left( \frac{\varphi ^{\prime }}{%
		\varphi }+2r_{1}\right),\text{ }c=n\left( n-1\right).
	\end{equation}
Since $a$ and $b$ are rational functions,  we obtain from \eqref{e22} and Lemma~\ref{liao} that
	\begin{equation}\label{e23}
	c\left( b^{2}-4ca\right) \frac{\varphi ^{\prime }}{\varphi }+b\left(
	b^{2}-4ca\right) -n\left( n-1\right) \left( b^{2}-4ca\right) ^{\prime
	}\equiv 0.
	\end{equation}
We now distinguish two sub-cases which correspond to the two sub-cases in Theorem~\ref{theo1.3}(2).

\begin{itemize}
\item[(i)]
	 If $b^{2}-4ca\not\equiv 0$, then by dividing both
	sides of \eqref{e23} by $b^{2}-4ca$ and substituting $b$, we get
	\begin{equation}\label{r1}
	r_1 = \frac{2n-1}{2} \frac{\left( b^{2}-4ac\right) ^{\prime }}{b^{2}-4ac}-n\frac{\varphi ^{\prime }}{\varphi }.
	\end{equation}
Since $a,b$ and $\varphi$ are rational, we obtain $r_1(z)=O(z^{-1})$, as $z\to\infty$. Furthermore, from \eqref{e21}, we have
	\begin{equation}\label{fRS}
	f'' +R(z) f' +S(z)f=0,
	\end{equation}
where
	\begin{align*}
	R(z)=- \left( \frac{n-1}{2n-1}\frac{\varphi ^{\prime }}{\varphi }-\frac{1}{2n-1}r_{1}\right) =O(z^{-1}), \quad \text{as}\; z\to \infty,
	\end{align*}
and $S(z)= -\frac{1}{2n-1}\frac{\psi }{\varphi }$ is a rational function. Suppose now that $S(z)\sim C_Sz^p$ as $z\to\infty$, where $C_S\in\C$. Applying Lemma~\ref{l2} on \eqref{fRS}, we have
	$p\geq -1$ and
	\begin{equation} \label{e24}
	\log M(r,f)=\frac{2\sqrt{\left| C_S\right|} }{p+2}r^{1+p/2}(1+o(1)),\quad r\to\infty.
	\end{equation}
	On the other hand, by applying Lemma \ref{l2} on \eqref{e2}, we get
	$$
	\log M(r,h)=\frac{2\sqrt{\left| C_{{0}}\right|} }{m+2}r^{1+m/2}(1+o(1)),\quad r\to\infty.
	$$
Since $\rho(f)=\rho(h)$, we have $p=m \geq -1$. 

Suppose first that $f$ has no poles. Dividing both sides of \eqref{e7} by $f^{2}$ and by applying Wiman-Valiron theory, we get
	\begin{equation*}
	C_{0}z^{m+2}+nC_{1} \nu(r) (1+o(1))+ n^{2} \nu(r)^{2} (1+o(1))=o(1),\quad r\not\in E,
	\end{equation*}
where $r=|z|$ and the set $E\subset[1,\infty)$ has finite logarithmic measure.
This leads to
	\begin{equation}
	\frac{C_{0}}{n^{2}}r^{m+2}+ \nu(r)^{2} (1+o(1))=o(1),\quad r\not\in E. \label{e25}
	\end{equation}
Applying the proof of Lemma~\ref{l2} to \eqref{e25}, we deduce that
	\begin{equation}
	\log M(r,f)=\frac{2\sqrt{\left| C_{0}\right| }}{n(m+2)}r^{1+m/2}(1+o(1)),\quad r\to\infty, \label{e26}
	\end{equation}
without an exceptional set.	Combining \eqref{e24} and \eqref{e26}, we get $\left| C_{0}\right| =n^{2}\left| C_S\right| $.

Suppose then that $f$ has poles, of which there can be at most finitely many by the assumption.
Then $g=f\Pi$ is entire, where $\Pi$ is a polynomial having zeros precisely at the poles of $f$.
Then the reasoning in the previous paragraph applies for $g$, and an analogue of \eqref{e26} holds
for $g$. Thus \eqref{e26} holds for $f$, and the conclusion $\left| C_{0}\right| =n^{2}\left| C_S\right| $
is valid in this case as well.
	
\item[(ii)]	 If $b^{2}-4ca\equiv 0$,  then \eqref{e22} can be written as
	\begin{align*}
	\varphi&= a f^2 +bff'+c(f')^2 = \frac{b^2}{4c} f^2 +bff'+c(f')^2\\
			&= c\left(\frac{b^2}{4c^2} f^2 +\frac{b}{c} ff'+(f')^2\right)=c\left( f'+\frac{b}{2c}f\right)^2\\
			&= n\left( n-1\right) \left( f^{\prime }+\frac{b}{2n\left( n-1\right) }f\right) ^{2}.
	\end{align*}
Set
	$$
	Q=f^{\prime }+\frac{b}{2n\left( n-1\right) }f.
	$$
Since $\varphi$ is rational, so is $Q$. By substituting $b\left(
	z\right) =\frac{n\left( n-1\right) }{2n-1}\left( \frac{\varphi ^{\prime }}{
		\varphi }+2r_{1}\right)$ and $\varphi =n\left( n-1\right)
	Q^{2}$ into the equation above, we get
	\begin{equation}\label{e27}
	f^{\prime }=\frac{-1}{2n-1}\left( \frac{Q^{\prime }}{Q}+r_{1}\right) f+Q.
	\end{equation}
Differentiating \eqref{e27} inductively, we obtain, for any integer $j\ge 1$, that
	\begin{equation}\label{e28}
	f^{(j)}=R_j f +S_j,
	\end{equation}
where $R_j$ and $S_j$ are rational functions. Using \eqref{e28} and the assumption $\gamma_P\leq n-2$, the differential polynomial $P(z,f)$ in \eqref{e1} can be rewritten as $P(z,f)=P_{n-2}(f)$,
	where $P_{n-2}(f)$ is a polynomial in $f$ of degree at most $n-2$ with rational coefficients. Thus, the equation \eqref{e1} becomes
	\begin{equation}\label{e29}
	h=f^n+P_{n-2}(f).
	\end{equation}
Differentiating \eqref{e29} twice and using \eqref{e27} and \eqref{e28} yields
	\begin{equation}
	h^{\prime }=\frac{-n}{2n-1}\left( \frac{Q^{\prime }}{Q}+r_{1}\right)
	f^{n}+nQf^{n-1}+P'_{n-2}(f),  \label{e30}
	\end{equation}%
and
	\begin{equation}\label{e31}
	\begin{split}
	h^{\prime \prime }=&\left( \frac{n^{2}}{\left( 2n-1\right) ^{2}}\left( \frac{%
		Q^{\prime }}{Q}+r_{1}\right) ^{2}-\frac{n}{2n-1}\left( \frac{Q^{\prime }}{Q}+r_{1}\right) ^{\prime }\right) f^{n}\\
		&+\left( nQ^{\prime }-nQ\left( \frac{%
		Q^{\prime }}{Q}+r_{1}\right) \right)
	f^{n-1}+n(n-1)Q^{2}f^{n-2}+P_{n-2}^{\prime \prime }\left( f\right).
	\end{split}
	\end{equation}
By substituting now \eqref{e29}-\eqref{e31} into \eqref{e2} and then using \eqref{e28} to get rid
of all derivatives of $f$, we obtain an equation of the form
	$$
	Af^n=P_{n-1}(f),
	$$
where $A$ is a rational function and $P_{n-1}(f)$ is a polynomial in $f$ of degree at most $n-1$.
Using Lemma~\ref{Mohon'ko lemma}, it follows that $A\equiv 0$, or alternatively
	\begin{equation}\label{e32}
	\left( \frac{n}{2n-1}\right) ^{2}\left( \frac{Q^{\prime }}{Q}+r_{1}\right)
	^{2}-\frac{n}{2n-1}\left( \frac{Q^{\prime }}{Q}+r_{1}\right) ^{\prime }-%
	\frac{n}{2n-1}\left( \frac{Q^{\prime }}{Q}+r_{1}\right) r_{1}=-r_{0}.
	\end{equation}
We claim that $r_1 \not\to 0$, as $ z\to \infty$. To prove this, we assume the contrary, i.e., there exists an integer $s\ge 1$ such that $|r_1(z)|=O(|z|^{-s})$ as $z\to \infty$. Then, from the equation \eqref{e32}, we obtain $|r_0(z)|=O\left(|z|^{-2}\right)$ as $z\to \infty$. But then Lemma~\ref{l2} implies that $h$ is rational, which is a contradiction. In particular, $r_1\not\equiv 0$. Now, dividing \eqref{e32} by $r_1 (Q'/Q+r_1)\not\equiv 0$ and letting $z\to\infty$, we obtain
	$$
	\lim_{z\rightarrow \infty }\frac{r_{0}\left( z\right) }{r_{1}^{2}\left(
		z\right) }=\frac{n\left( n-1\right) }{\left( 2n-1\right) ^{2}}.
	$$
Recalling \eqref{global}, this gives us $m=2l\geq 0$ and
	$$
	|C_1|=(2n-1)\sqrt{\frac{|C_0|}{n(n-1)}}.
	$$
Also recall that $f$ satisfies \eqref{e27}, where
	$$
	S(z)=\frac{1}{2n-1}\left( \frac{Q^{\prime}(z)}{Q(z)}+r_{1}(z)\right).
	$$
Since $Q$ is rational and $r_1 \not\to 0$, we obtain
$S(z)\sim \frac{r_1(z)}{2n-1}\sim\frac{C_{1}}{2n-1}z^l$ as $z\to\infty$.
Hence, Lemma~\ref{l3} gives us
	\begin{eqnarray*}
	\log M(r,f) &=&\frac{\left| C_1\right| }{(l+1)(2n-1)}r^{1+l}(1+o(1))\\
	&=&\frac{2\sqrt{\left|C_0\right|}}{(m+2)\sqrt{n(n-1)}}r^{1+\frac{m}{2}}(1+o(1)),
	\quad r\to\infty.
	\end{eqnarray*}
This yields the assertion in the sub-case (ii).
\end{itemize}

\subsection{\sc Proof of Theorem \ref{theo1.3}: The case of infinitely many poles}\label{Case2}

The case when $f$ has finitely many poles was proved in Section~\ref{Case1}. In this section
we suppose that $f$ has infinitely many poles such that $N(r,f)=S(r,f)$, and aim for
a contradiction.  

Let $\varphi$ be the function defined in \eqref{e7}.
If $\varphi \equiv 0$, then it follows from the proof of Theorem~\ref{theo1.0} that $f$ has
finitely many zeros and poles, which is a contradiction. Hence we may suppose that $\varphi \not\equiv 0$.
A simple modification of Case~(2) in the proof of Theorem~\ref{theo1.0} gives us
	\begin{equation}\label{varphi-small}
	T(r,\varphi)=S(r,f).
	\end{equation}

Next we will follow the reasoning in Section~\ref{Case1}. We find that the meromorphic function $\psi$ defined in \eqref{e20} satisfies
	$$
	T(r,\psi)=S(r,f),
	$$
and that $f$ solves the second order differential equation \eqref{e21}. 
The coefficients $a$, $b$ in \eqref{abc} are now meromorphic functions satisfying
	$$
	T(r,a)=S(r,f)\quad\textnormal{and}\quad T(r,b)=S(r,f),
	$$ 
and hence they are small functions of $f$. 
Suppose that $f$ has a pole at $z_0$ of multiplicity $k$, which is not a zero or a pole of $r_i(z)$ ($i=0,1$). Then $f$ has the Laurent expansion \eqref{Laurent} in the neighbourhood of $z_0$.
Using Lemma~\ref{liao}, this discussion leads to \eqref{e23}, where we distinguish two sub-cases.

	\begin{itemize}
	\item[(i)] If $b^{2}-4ca\not\equiv 0$, we get \eqref{r1}.
	Then, for all $z$
	near $z_0$, it is easy to see that
	\begin{eqnarray*}
	\varphi(z) &=& \left(r_0f^2+nr_1f'f+nf''f+n(n-1)(f')^2\right)(z)\\
	&=&\frac{kn(kn+1)a_k^2}{(z-z_0)^{2k+2}}+O\left(\frac{1}{(z-z_0)^{2k+1}}\right),\\
	\left(\frac{\psi}{\varphi}\right)(z) &=& \left(r_1\frac{f'}{f}-(n-1)\frac{\varphi'}{\varphi}\frac{f'}{f}+(2n-1)\frac{f''}{f}\right)(z)\\
	&=& \frac{k(k+1)}{(z-z_0)^2}+O\left(\frac{1}{z-z_0}\right).
	\end{eqnarray*}
	This gives raise to
	\begin{eqnarray*}
	a(z) &=& \left(r_0+\frac{n}{2n-1}\frac{\psi}{\varphi}\right)(z)=\frac{kn(k+1)}{2n-1}\frac{1}{(z-z_0)^2}+O\left(\frac{1}{z-z_0}\right),\\
	b(z) &=& \frac{n(n-1)}{2n-1}\left(\frac{\varphi'}{\varphi}+2r_1\right)(z)=\frac{-2n(n-1)(k+1)}{2n-1}\frac{1}{z-z_0}+O(1).
	\end{eqnarray*}
	It is therefore easy to see that $b^2-4ac$ has a double pole at $z_0$, so that
		$$
		\frac{(b^2-4ac)'}{b^2-4ac}=\frac{-2}{z-z_0}+O(1)
		$$
	near $z_0$. Finally, the rational function $r_1$ in \eqref{r1} satisfies
		$$
		r_1(z)=\left(\frac{2n-1}{2} \frac{\left( b^{2}-4ac\right) ^{\prime }}{b^{2}-4ac}-n\frac{\varphi ^{\prime }}{\varphi }\right)(z)=\frac{2kn+1}{z-z_0}+O(1)
		$$
	near $z_0$. Since $2kn+1\neq 0$, this means that $r_1$ has a simple pole at $z_0$, which contradicts
	our assumption on $z_0$.
	\item[(ii)]  If $b^{2}-4ca\equiv 0$, we get \eqref{e27}, where $T(r,Q)=S(r,f)$ because of \eqref{varphi-small}. Denote $F=Q'/Q+r_1$ for short. It is easy to see that
		\begin{eqnarray*}
		Q(z)&=&\left(f'+\frac{b}{2n(n-1)}f\right)(z)\\
		&=&-\frac{(2kn+1)a_k}{2n-1}\frac{1}{(z-z_0)^{k+1}}
		+O\left(\frac{1}{(z-z_0)^{k}}\right)
		\end{eqnarray*}
	near $z_0$. In particular, $Q$ has a pole of multiplicity $k+1$ at $z_0$. Thus
		$$
		F(z)=-\frac{k+1}{z-z_0}+O(1),
		$$
	which implies that
		\begin{eqnarray*}
		&&\left(\left(\frac{n}{2n-1}\right)^2F^2-\frac{n}{2n-1}F'-\frac{n}{2n-1}r_1F\right)(z)\\
		&&\qquad\qquad =\frac{n(k+1)(kn-n+1)}{(2n-1)^2}\frac{1}{(z-z_0)^2}+O\left(\frac{1}{z-z_0}\right).
		\end{eqnarray*}
	However, from \eqref{e32} it now follows that $r_0$ has a double pole at $z_0$, which  contradicts
	our assumption on $z_0$.
	\end{itemize}

From the sub-cases (i) and (ii) we conclude that all possible poles of $f$ must be among the
zeros and poles of the rational coefficients $r_i(z)$ ($i=0,1$). But this contradicts the
assumption that $f$ has infinitely many poles.


\section{Proofs of corollaries}\label{proofs-cor-sec}


\subsection{\sc Proof of Corollary \ref{theo1.1}}

We prove that under either one of the assumptions (1) and (2), the function $\vp$ defined in \eqref{e7} vanishes identically. Suppose that $\vp\not\equiv 0$. From the proof of Theorem~\ref{theo1.0}, $ f $ satisfies Case (2) of Theorem~\ref{theo1.0}.

Under the assumption (1), it follows that for any $\veps \in (0,1)$ sufficiently small, there is an $R>0$ such that
	$$
	\frac2j\overline{N}(r,1/f) + N(r,f) + \frac2j\overline{N}(r,f)\le  (1-\varepsilon) T(r,f),\quad r>R.
	$$
This together with Case (2) of Theorem~\ref{theo1.0} gives $T(r,f)=S(r,f)$, which is absurd.

If the assumption (2) holds, we have $\gamma_Q \leq \Gamma_Q \leq n-3$
since $\gamma_P \leq \Gamma_P \leq n-5$. Therefore, by
Lemmas~\ref{Clunie} and~\ref{Clunie-Doeringer}, we deduce from \eqref{phij} when $ j=2,3$ that $T(r,\varphi)=T(r,\varphi_2)=S(r,f)$ and $T(r,\varphi_3)=S(r,f)$. Then
$$
T(r,f)=T\left( r,\dfrac{\varphi_{3}}{\varphi}\right) =S(r,f),
$$
which is a contradiction.

Thus, either one of the assumptions (1) or (2) implies $\varphi\equiv 0$, and hence the assertion 
follows from the proof of Theorem~\ref{theo1.0}.

\subsection{\sc Proof of Corollary \ref{CC-corollary}: The linearly independent case}\label{linearly-independent}

Suppose that the functions $h_1,h_2$ in \eqref{hee} 
are linearly independent. Then their Wronskian $W$ satisfies 
	\begin{equation}\label{wr}
	W=W\left(h_1,h_2\right)=-p_{1}p_{2}\left( \frac{p_{1}^{\prime }}{p_{1}}-\frac{p_{2}^{\prime }}{p_{2}}+\alpha_{1}^{\prime }-\alpha_{2}^{\prime }\right) e^{\alpha_1(z)+\alpha_2(z)} \not\equiv 0,
	\end{equation}
and the functions $h_1,h_2$ solve \eqref{e2}, where $r_2(z)\equiv 0$ and
	\begin{eqnarray*}
	r_{1}(z)&=&-(\alpha_{1}^{\prime }+\alpha_{2}^{\prime })-\frac{p_{1}^{\prime }}{p_{1}}-
	\frac{p_{2}^{\prime }}{p_{2}}-\frac{\left( \frac{p_{1}^{\prime }}{p_{1}}-
	\frac{p_{2}^{\prime }}{p_{2}}+\alpha_{1}^{\prime }-\alpha_{2}^{\prime }\right)^{\prime }}
	{\frac{p_{1}^{\prime }}{p_{1}}-\frac{p_{2}^{\prime }}{p_{2}}
	+\alpha_{1}^{\prime }-\alpha_{2}^{\prime }}\\
	r_{0}(z)&=&\alpha_{1}^{\prime }\alpha_{2}^{\prime }+\frac{p_{1}^{\prime }}{p_{1}}
	\alpha_{2}^{\prime }+\frac{p_{2}^{\prime }}{p_{2}}\alpha_{1}^{\prime }+\frac{
	p_{1}^{\prime }}{p_{1}}\frac{p_{2}^{\prime }}{p_{2}}\\
	&&+\left( \frac{
	p_{2}^{\prime }}{p_{2}}+\alpha_{2}^{\prime }\right) \left( -\frac{\left( \frac{
	p_{2}^{\prime }}{p_{2}}+\alpha_{2}^{\prime }\right) ^{\prime }}{\frac{
	p_{2}^{\prime }}{p_{2}}+\alpha_{2}^{\prime }}
	+\frac{\left( \frac{p_{1}^{\prime }}{p_{1}}-\frac{p_{2}^{\prime }}
	{p_{2}}+\alpha_{1}^{\prime }-\alpha_{2}^{\prime }\right)^{\prime }}
	{\frac{p_{1}^{\prime }}{p_{1}}-\frac{p_{2}^{\prime }}{p_{2}}
	+\alpha_{1}^{\prime }-\alpha_{2}^{\prime }}\right) 
	\end{eqnarray*}
are well-defined rational functions \cite[Proposition~1.4.6]{L}. 

Next, we prove that $s_1=s_2$. Suppose on the contrary to this claim that
$s_1\neq s_2$. Then clearly $|r_{1}(z)|\rightarrow \infty $, as $|z|\rightarrow +\infty $, thus the conclusion in Case (i) of Theorem~\ref{theo1.3}(2) cannot hold. 
To exclude Case (ii) of Theorem~\ref{theo1.3}(2), we
notice that 
	\begin{eqnarray*}
	m &=&\deg _{\infty }(r_{0})=s_{1}+s_{2}-2, \\
	l &=&\deg _{\infty }(r_{1})=\max \{s_{1}-1,s_{2}-1\}.
	\end{eqnarray*}
If $m=2l$ holds, then 
	\[
	s_{1}+s_{2}-2=2\max \{s_{1}-1,s_{2}-1\}=s_{1}+s_{2}-2+|s_{1}-s_{2}|,
	\]
i.e., $s_{1}=s_{2}$, which is a contradiction.  Thus it remains to prove that Theorem~\ref{theo1.3}(1) cannot hold. Indeed, if $f(z)=q(z)e^{\alpha (z)}$, then from \eqref{characteristics} we get $\deg (\alpha )=\max \{s_{1},s_{2}\}$, say $\deg(\alpha)=s_{1}$.
Substituting this $f$ in \eqref{e1} yields an equation of the form \eqref{reduced-eqn}, where the
dominating terms $q^ne^{n\alpha}$ and $p_1e^{\alpha_1}$ must cancel out. The remainder of \eqref{reduced-eqn} is clearly impossible. This completes the proof of $s_1=s_2$. 

The remainder of the proof is divided to the three possible situations in Theorem~\ref{theo1.3}.

\medskip
(1) If Case (1) of Theorem~\ref{theo1.3} occurs, then $f$ has the
form $q(z)e^{\alpha(z) }$, where $q$ is non-zero rational function and $\alpha $ is non-constant polynomial of degree $s=s_1=s_2$, whose leading  coefficient is some constant $a\in\C\setminus\{0\}$. From
\cite{Stein}, we have 	
	\begin{eqnarray*}
	T(r,h)&=&(|a_1|+|a_2|+|a_1-a_2|) \frac{r^s}{2\pi} (1+o(1)),\\
	T(r,f)&=&|a| \frac{r^s}{\pi}(1+o(1)).
	\end{eqnarray*}
Hence \eqref{characteristics} yields	
		\begin{eqnarray}
		n|a| & = & \frac{|a_1|+|a_2|+|a_1-a_2|}{2} \label{na}\\
			&\ge & \frac{|a_1|+|a_2|+\left| |a_1|-|a_2| \right|}{2} \nonumber\\
			& = &  \max\left\{|a_1|,|a_2|\right\} \overset{\textnormal{set}}{=} |a_1|. \nonumber
		\end{eqnarray}
Suppose that $n|a|>|a_1|$, and let $\eta_1 = \frac{a_1}{na}$ and $\eta_2 =\frac{a_2}{na}$. Then $0<|\eta_2|\le|\eta_1|<1$, and \eqref{na} can be re-written as $2 = |\eta_1| + |\eta_2| +|\eta_1 -\eta_2|$,
which is equivalent to
	$$	
	2-2(|\eta_1| + |\eta_2|)+|\eta_1| |\eta_2| (1+\cos \varphi) =0,\quad \varphi = \arg (\eta_1 /\eta_2).
	$$
Solving this equation for $|\eta_1|$ yields
		$$
		|\eta_1| = \frac{2|\eta_2|-2}{|\eta_2|(1+\cos\varphi) -2}>1,
		$$
which is a contradiction. Thus $n|a|=|a_1|\geq |a_2|$, which leads us to consider three cases. 
\begin{itemize}
\item[(i)] If $a_1=a_2$, then the dominating terms $q^ne^{n\alpha}$
and $h$ in \eqref{reduced-eqn} must cancel out. This is possible only when $e^{\alpha_1}=e^{\alpha_2}$
because $\alpha_1(0)=0=\alpha_2(0)$ by the assumption. Thus $q^n=p_1+p_2$, and, consequently, $P(z,f)\equiv 0$. In particular, if $p_1,p_2$ are polynomials, then $q$ is also a polynomial.

\item[(ii)] If $|a_1|=|a_2|$ but $a_1\neq a_2$, then \eqref{reduced-eqn} has three dominating
terms $q^ne^{n\alpha},h_1,h_2$ of order~$s$. Due to the exponential factors, each of these
three terms tend exponentially to zero/infinity in sectors that cover $\C$ and may overlap. 
However, it is not possible
that all three terms would cancel out everywhere in $\C$ because of our assumption $|a_1|=|a_2|$, $a_1\neq a_2$. Hence this case is impossible. 

\item[(iii)] If $|a_1|>|a_2|$, then $n\alpha =\alpha_1$ from $n|a|=|a_1|$ and \eqref{reduced-eqn}. In other words,
$q^ne^{n\alpha}$ and $h_1$ are dominating terms in \eqref{reduced-eqn}, which cancel out. This
leaves us with the reduced equation $P(z,f)\equiv h_2$. 
\end{itemize}

\medskip
(2) If the sub-case (i) of Theorem~\ref{theo1.3}(2) occurs, then $l\le -1 \le m$, which implies $\alpha_1'+\alpha_2'=0$, or, in other words $\alpha_1+\alpha_2 \equiv 0$ by $\alpha_1(0)=0=\alpha_2(0)$.  In this case it follows from \eqref{wr} that $W$ is rational. By a simple calculation, or by using \cite[Proposition~1.4.6]{L}, we
infer $r_1=-W'/W$. Thus, from \eqref{r1}, we get
	\[
	-2\frac{W^{\prime }}{W}=(2n-1)\frac{\left( b^{2}-4ac\right) ^{\prime }}{\left( b^{2}-4ac\right) }-2n\frac{\varphi ^{\prime }}{\varphi },
	\]
which leads to
	\[
	\sqrt{b^{2}-4ac}=C\frac{W\left( b^{2}-4ac\right) ^{n}}{\varphi ^{n}},\quad C\in\C\setminus\{0\}.
	\]
Since $W,a,b$ and $\varphi $ are rational, it follows that $\sqrt{b^{2}-4ac}$
is also rational. Thus  \eqref{e22} can be re-written as
	\[
	\frac{\varphi }{n(n-1)}=\left( f^{\prime }-R_{-}f\right) \left( f^{\prime
	}-R_{+}f\right),
	\]
where $R_{\pm }=\displaystyle\frac{-b\pm \sqrt{b^{2}-4ac}}{2c}$ are rational functions. Therefore, 
	\begin{equation}\label{R-plusminus}
	\begin{split}
	f^{\prime }-R_{-}f &=\Pi _{1}e^{\beta (z)}, \\
	f^{\prime }-R_{+}f &=\Pi _{2}e^{-\beta (z)},
	\end{split}	
	\end{equation}
where $\Pi _{1},\Pi _{2}$ are rational functions and  $\beta $ is a polynomial. By subtracting the lower equation from the upper equation in \eqref{R-plusminus}, we get $f=q_{1}e^{\beta (z)}+q_{2}e^{-\beta (z)}$, where $q_1,q_2$ are rational functions. Substituting this form of $f$ into \eqref{e1}, we obtain four dominating terms 
$q_1e^{n\beta}, q_2e^{-n\beta}, h_1, h_2$, which must cancel out in pairs. 
This happens only when $n\beta =\pm\alpha_1$ and $\alpha_1=-\alpha_2$ hold.

(3) If the sub-case (ii) of Theorem~\ref{theo1.3}(2) holds, then $m=2l$ and  $C_{0} = \frac{n(n-1)}{(2n-1)^{2}}C_{1}^{2}$. We have
	\begin{eqnarray*}
	r_{1}(z) &=& -s(a _{1}+a _{2})z^{s-1}+O(z^{s-2}), \\
	r_{0}(z) &=& s^{2} a _{1} a _{2}z^{2s-2}+O(z^{s-2}).
	\end{eqnarray*}
Thus $C_{0} = \frac{n(n-1)}{(2n-1)^{2}}C_{1}^{2}$ is equivalent to
	$$	
	\left( \frac{ a _{1}}{ a _{2}}\right) ^{2}+\left(2- \frac{(2n-1)^2}{n(n-1)}\right) \left( \frac{ a _{1}}{ a _{2}}\right) +1=0,
	$$
which in turn is equivalent to $\frac{ a _{1}}{ a _{2}}\in \left\{ \frac{n}{n-1},\frac{n-1}{n}
\right\}$. Since we assumed that $| a _{1}|\geq | a _{2}|$,  the only possibility is 
$\frac{ a _{1}}{ a _{2}}=\frac{n}{n-1}$. 
Solving \eqref{e27}, we obtain $f=q_{1}e^{\beta }+q_{2}$, where $q_{1},q_{2}$ are rational functions 
and $\beta$ is a polynomial. The term $f_{0}=q_{1}e^{\beta }$ satisfies the homogeneous
equation corresponding to \eqref{e27}, i.e.,
	$$
	(2n-1)\frac{f_{0}^{\prime }}{f_{0}}=\frac{W^{\prime }}{W}-\frac{Q^{\prime }}{Q},
	$$
because $r_1=-W'/W$. Integration on both sides results in
	\[
	f_{0}^{2n-1}=C\frac{W}{Q}=p_{0}e^{\alpha_{1}+\alpha_{2}},\quad C\in\C\setminus\{0\},
	\]
where $p_{0}$ is a rational function. Thus $\beta (z)=\frac{\alpha_{1}(z)+\alpha_{2}(z)}{2n-1}$.

\subsection{\sc Proof of Corollary \ref{CC-corollary}: The linearly dependent case}\label{linearly-dependent}

Suppose that the functions $h_1,h_2$ in \eqref{hee} 
are linearly dependent. Then $h$ is of the form $h(z)=p(z)e^{\gamma(z)}$, where $p$ is
rational and $\gamma$ is a non-constant polynomial. We see that $h$ solves \eqref{e2}, where 
$r_1(z)\equiv r_2(z)\equiv 0$ and
	$$
	-r_0(z)=\frac{p''(z)}{p(z)}+2\gamma'(z)\frac{p'(z)}{p(z)}+\gamma'(z)^2
	$$
is a rational function. 


It is easy to see that the sub-case (ii) in Theorem~\ref{theo1.3}(2) cannot hold. Next we show
that the sub-case (i) in Theorem~\ref{theo1.3}(2) cannot hold either. Since $r_1\equiv 0$, 
equation \eqref{r1} reduces to
	\[
	(2n-1)\frac{\left( b^{2}-4ac\right) ^{\prime }}{\left( b^{2}-4ac\right) }=2n\frac{\varphi ^{\prime }}{\varphi },
	\]
which leads to
	\[
	\sqrt{b^{2}-4ac}=C\left( \frac{b^{2}-4ac}{\varphi}\right)^{n},\quad C\in\C\setminus\{0\}.
	\]
Similarly as in Section~\ref{linearly-independent}(2), we deduce that $f=q_{1}e^{\beta (z)}+q_{2}e^{-\beta (z)}$, where $q_1,q_2$ are rational functions and $\beta$ is a polynomial. Substituting this form of $f$ into \eqref{e1}, we obtain three dominating terms 
$q_1e^{n\beta}, q_2e^{-n\beta}, pe^\gamma$, which cannot cancel out.

It follows that the situation in Theorem~\ref{theo1.3}(1) holds. But now the discussion in
Section~\ref{linearly-independent}(1) applies, and we are done.

\bigskip
\noindent
\textbf{Acknowledgements.}
Latreuch was supported by the University of Mostaganem. Wang was supported by the Natural Science Foundation of China (No.~11771090) and the Natural Science Foundation of Shanghai (No.~17ZR1402900).
They both want to thank the Department of Physics and Mathematics at the UEF for
its hospitality during their visit.



\end{document}